\documentclass[a4paper,11pt,reqno]{amsart}
\usepackage[utf8]{inputenc}
\usepackage[T1]{fontenc}

\usepackage{paralist}
\usepackage{letterspace}
\usepackage{booktabs}
\usepackage{ifdraft}
\usepackage{cite}
\usepackage{microtype}
\usepackage{tikz,tikz-cd}
\usepackage{lmodern}
\usepackage{comment}
\usepackage{xcolor}
\ifdraft{\usepackage[notcite,notref]{showkeys}}{}

\usepackage{amsmath,amssymb,amsthm}
\usepackage{mathtools,mathrsfs}
\usepackage{dsfont}
\usepackage{stmaryrd}

\usepackage[hidelinks]{hyperref}

\calclayout

\newcommand{\IC}{\mathbb{C}}
\newcommand{\IE}{\mathbb{E}}
\newcommand{\IN}{\mathbb{N}}
\newcommand{\IP}{\mathbb{P}}
\newcommand{\IR}{\mathbb{R}}
\newcommand{\IZ}{\mathbb{Z}}

\newcommand{\one}{\mathds{1}}

\newcommand{\Geo}{\mathsf{Geo}}
\newcommand{\Ber}{\mathsf{Ber}}
\newcommand{\Exp}{\mathsf{Exp}}
\newcommand{\Poi}{\mathsf{Poi}}

\newcommand{\Beta}{\operatorname{B}}

\newcommand{\mc}{\mathcal}

\newcommand{\ga}{\alpha}
\newcommand{\gb}{\beta}
\newcommand{\gC}{\Gamma}
\newcommand{\gc}{\gamma}

\newcommand{\gep}{\varepsilon}
\newcommand{\gk}{\kappa}

\newcommand{\gl}{\lambda}

\newcommand{\gS}{\Sigma}

\newcommand{\gT}{\Theta}
\newcommand{\go}{\omega}

\newcommand{\dto}{\stackrel{\mathrm{d}}{\longrightarrow}}
\newcommand{\pto}{\stackrel{\mathbb{P}}{\longrightarrow}}

\newcommand{\ssq}{\subseteq}
\newcommand{\Di}{\,\mathrm{d}}
\newcommand{\floor}[1]{\lfloor #1 \rfloor}

\newcommand{\descend}{\stackrel{*}{\longrightarrow}}

\newcommand{\ER}{Erd\H{o}s--R\'enyi\ }

\mathtoolsset{showonlyrefs,showmanualtags}

\newtheorem{lemma}{Lemma}
\newtheorem{theorem}[lemma]{Theorem}

\theoremstyle{definition}

\theoremstyle{remark}
\newtheorem{remark}[lemma]{Remark}

\title{Ancestries in random $d$-DAGs}
\author[F. Burghart]{Fabian Burghart}
\address{Institute for Discrete Mathematics, TU Graz, 8010 Graz, Austria}
\email{fabian.burghart@tugraz.at}
\thanks{The author wishes to thank Svante Janson for his insightful comments on the manuscript, in particular for pointing out a simplification to the proof of Theorem~\ref{thm:technical}. This research was funded in part by the Austrian Science Fund (FWF) [10.55776/F1002]}
\date{}
\keywords{}
\subjclass[2020]{60C05 (Primary); 60F05, 05C20, 05C80 (Secondary)}

\begin{document}

\begin{abstract}
 We consider a random recursive DAG $G_n$ on the vertex set $[n]$ where every vertex $i\geq 2$ has out-degree $d$, with the targets chosen uniformly at random among the earlier $i-1$ vertices. For this model, we propose a novel way to investigate the descendants of $n$ (which have recently been studied in a paper by Janson) through what we call \emph{ancestry processes}. The ancestor process $a_i(n)$ of a vertex $i$ is defined as the number of ancestors of $i$ in $G_n$, and is closely related to the evolutions of multi-draw P\'olya urns. Results on the descendants can then be obtained via asymptotic results on functionals of the ancestry processes, generally leading to technical integral expressions. This method yields the answer to two questions posed by Janson, the first on the size of the joint descendants of vertices $n$ and $n+1$, and the other on the location of the earliest non-descendant. We further prove limit theorems for the ancestry processes $a_i(n)$ depending on $i$, determine the location of the earliest source node, and provide an alternative proof of a first-moment result contained in Janson's work.
\end{abstract}

\maketitle


\section{Introduction and Main Results}
\label{sec:intro}

Denote by $G_n$ the random \emph{directed acyclic graph} (or \emph{DAG}, for short) on the vertex set $[n]$ constructed as follows: For each $2\leq j\leq n$, we draw $d\geq 1$ arrows starting in vertex $j$ to vertices in $[j-1]$ chosen uniformly at random and independently of everything else.
In particular, this means that vertex $1$ is the unique sink of the DAG, vertex $n$ is a source (but possibly not the only one), and all vertices $2\leq j\leq n$ have outdegree $d$. We also note that we can couple the $G_n$ together to a random DAG-process, by adding new vertices one at a time and connecting them with $d$ outgoing arrows to the previous graph. In fact, we will extensively rely on this construction as a \emph{recursive $d$-DAG}.

For $d=1$, this DAG yields the well-studied \emph{random recursive tree}, see e.g. \cite[Chapter~6]{Drm2009}, up to our unusual orientation of the arrows. We will assume $d\geq 2$ throughout. This case too has received significant attention over the past 30 years, often motivated by interpreting the DAG as a random circuit. Indeed, questions regarding the depth and distances between vertices have been studied extensively, as in \cite{AGM1999,BF2012,DJ2011,TX1996}. The number and distribution of source nodes, corresponding to output gates in the random circuit model, has been investigated in \cite{KS2017,MPSU2014,TM2001}; more general questions about vertices of fixed degree or the maximal degree have been addressed in \cite{MT2002} or \cite{DL1995}, respectively.    More recently, Stein's method was successfully employed to derive normal approximations for subgraph counts in $G_n$, in \cite{BHJL2025}.

This paper continues a line of inquiry started in \cite{Jan2023}, which requires some notation. We write $i\longrightarrow j$ to indicate that there is an arrow from $i$ to $j$, and we write $i\descend j$ to indicate that $j$ can be reached from $i$ by a directed path.
For $n\in \IN$, let
\begin{equation}\label{eq:defDn}
 D_n:=\{j\in[n]: n\descend j\}
\end{equation}
denote the set of \emph{descendants} of $n$; and for $i\in[n]$, let
\begin{equation}\label{eq:defAn}
 A_i(n):=\{j\in[n]:j\descend i\}
\end{equation}
denote the set of \emph{ancestors} of $i$ in $G_n$. It will be notationally convenient later to also allow $i>n$, in which case of course $A_i(n)=\emptyset$.

Janson's result in \cite{Jan2023}, motivated by \cite[Problem~281]{Knu2024}, was that the limiting distribution of $|D_n|$ under suitable rescaling is a power of a gamma distribution; to be precise, he obtained the following theorem through a combination of a graph exploration and sophisticated martingale methods:
\begin{theorem}[{\cite[Theorem~1.4]{Jan2023}}]\label{thm:Janson}
 Let $d\geq 2$. Then, as $n\to\infty$, both in distribution and for all moments,
 \begin{equation}\label{eq:Janson}
  n^{-\frac{d-1}{d}}|D_n| \to \frac{(d-1)^{1/d}\pi}{d\sin(\pi/d)} \gc^{1/d},
 \end{equation}
 where $\gc$ is a gamma-distributed random variable with parameters $\frac{d}{d-1}$ and 1, thus having density $\gC\left(\frac{d}{d-1}\right)^{-1} x^{-\frac{1}{d-1}} e^{-x}$ for $x>0$. In particular, this means
 \begin{equation}\label{eq:thmSvante}
  n^{-\frac{d-1}{d}} \IE |D_n|
  \to \frac{(d-1)^{1/d} \pi}{d\sin(\pi/d)} \cdot \frac{\gC\left(\frac{d}{d-1}+\frac{1}{d}\right)}{\gC\left(\frac{d}{d-1}\right)} \qquad \text{as }n\to\infty.
 \end{equation}
\end{theorem}

Subsequently, Janson posed the following two problems:
\begin{enumerate}[(i)]
 \item\label{item:q1} Is it true that $n^{-\frac{d-1}{d}} \IE|D_n\cap D_{n+1}|$ converges to a constant $\nu>0$? What is $\nu$? What is the asymptotic distribution of $n^{-\frac{d-1}{d}} |D_n\cap D_{n+1}|$, assuming it exists? (\cite[Problem~9.7]{Jan2023}, albeit only for $d=2$)
 \item\label{item:q2} What is the smallest $i\notin D_n$? (During his plenary talk at the Random Structures and Algorithms conference in Vienna, August 2025).
\end{enumerate}
The motivation of this paper is to answer both of these questions -- we give a closed expression for $\nu$ in Theorem~\ref{thm:jointdesc} and a threshold result for the earliest non-descendant in Theorem~\ref{thm:threshold}.

The proofs of these theorems rely on the \emph{ancestry process} of a given vertex $i$, which we will define as $a_i(n):=|A_i(n)|$, combined with the simple observation that
\begin{equation}\label{eq:keyobs}
 i\in D_n \quad \text{ if and only if } \quad a_i(n)=1+a_i(n-1).
\end{equation}
It turns out that the ancestry process has several nice properties, which we will obtain in Section~\ref{sec:ancestry}, and which will allow us to write expectations such as $\IE |D_n|$ or $\IE |D_n\cap D_{n+1}|$ as explicit sums over probabilities of events for ancestry processes; cf. Lemmas~\ref{lemma:aindistr} and \ref{lemma:firstmoment} below. In Section~\ref{sec:technical} we prove a technical theorem (Theorem~\ref{thm:technical}) to evaluate the asymptotics of such sums  in terms of integrals over a first-order approximation; and then evaluate the resulting integrals in Section~\ref{sec:integrals} to obtain the limit in \eqref{eq:thmSvante} and the constant $\nu$ in Theorem~\ref{thm:jointdesc}.

Our methods also yield limit laws for $a_i(n)$ for the full range of $1\leq i\leq n$ in Theorem~\ref{thm:limit}, as well as a threshold result for the earliest source node, cf. Theorem~\ref{thm:sources}.

We suspect that our proof methods, after overcoming significantly more technical difficulties, will also give the convergence of higher moments needed to turn Theorem~\ref{thm:jointdesc} into a statement on convergence in distribution through the method of moments, thus answering Problem (\ref{item:q1}) in its entirety. It might even extend to the related model of DAGs obtained through preferential attachment as in \cite{JL2024}, but we have not pursued this line of inquiry.

\subsection{Main results}

\begin{theorem}\label{thm:jointdesc}
 For fixed $d\geq 2$, we have
 \begin{equation}\label{eq:thmjointdesc}
  n^{-\frac{d-1}{d}} \IE |D_n\cap D_{n+1}|
  \to \frac{(d-1)^{1/d} \pi}{d\sin(\pi/d)} \left[ \frac{2\gC\left(\frac{d}{d-1}+\frac{1}{d}\right)}{\gC\left(\frac{d}{d-1}\right)} - \frac{\gC\left(\frac{2d}{d-1}+\frac{1}{d}\right)}{\gC\left(\frac{2d}{d-1}\right)} \right]
  \qquad \text{as }n\to\infty.
 \end{equation}
\end{theorem}

For example, in the special case $d=2$, the moment convergence in \eqref{eq:Janson} yields $n^{-1/2} \IE |D_n| \to 3\pi^{3/2}/8$, whereas we obtain from \eqref{eq:thmjointdesc} that
\[
 n^{-1/2} \IE |D_n\cap D_{n+1}| \to \frac{13\pi^{3/2}}{64}.
\]

\begin{remark}
 It follows from the proof of Theorem~\ref{thm:jointdesc} that the expected number of joint descendants of $m$ different vertices, all within $\{n,n+1,\dots,n+M\}$ where $M$ is a constant, scaled by $n^{-(d-1)/d}$, converges to
 \begin{equation}\label{eq:multijoint}
  \frac{(d-1)^{1/d} \pi}{d\sin(\pi/d)} \sum_{i=1}^m (-1)^{1+i}\binom{m}{i} \frac{\gC\left(\frac{id}{d-1}+\frac{1}{d}\right)}{\gC\left(\frac{id}{d-1}\right)}.
 \end{equation}
\end{remark}

\begin{remark}
 After noting that $\IE|D_n\cap D_{n+1}| = \IE|D_n| + \IE|D_{n+1}| - \IE|D_n\cup D_{n+1}|$ and applying \eqref{eq:thmSvante}, a comparison with \eqref{eq:thmjointdesc} reveals that 
 \begin{equation}
  n^{-\frac{d-1}{d}} \IE|D_n\cup D_{n+1}| \to \frac{(d-1)^{1/d}\pi}{d\sin(\pi/d)} \frac{\gC\left(\frac{2d}{d-1}+\frac{1}{d}\right)}{\gC\left(\frac{2d}{d-1}\right)}.
 \end{equation}
 A more general application of the inclusion-exclusion principle to \eqref{eq:multijoint} inductively yields an analogous interpretation to the summands there. 
\end{remark}

\begin{remark}
 In contrast to Theorem~\ref{thm:jointdesc}, the number of joint descendants for $d=1$ converges to a geometric distribution on the positive integers with $p=\frac{1}{2}$; cf. \cite[Theorem~7]{KW2010}.
\end{remark}

As an offshoot of the proof of Theorem~\ref{thm:jointdesc}, we also obtain an alternative proof of \eqref{eq:thmSvante} in Section~\ref{ssec:EDn}, as well as the following theorem on the ancestry of an arbitrary vertex $i$:

\begin{theorem}\label{thm:limit}
 As $n\to\infty$, the number of ancestors of $i$ in $G_n$ satisfies the following:
 \begin{enumerate}[(i)]
  \item\label{it:prop1} For $i=o\big(n^{(d-1)/d}\big)$, we have $n^{-1} a_i(n)\pto 1$.
  \item\label{it:prop2} For $i\sim \ga n^{(d-1)/d}$ with $\ga>0$, we have $n^{-1} a_i(n) \dto \Xi_\ga$, where $\Xi_\ga$ is a random variable on $[0,1]$ with cumulative distribution function
  \begin{equation}\label{eq:thmlimitXi}
   \IP(\Xi_\ga\leq \gc) = 1 - e^{-\frac{\ga^d}{d-1}\big((1-\gc)^{1-d}-1\big)}.
  \end{equation}
  Moreover, the convergence also holds for all moments.
  \item\label{it:prop3} For $i=\go\big(n^{(d-1)/d}\big)$ and $i=o(n)$, we have $(i/n)^d a_i(n)\dto \Exp(1)$, where the convergence also holds for all moments.
  \item\label{it:prop4} For $i\sim \gb n$ with $0<\gb \leq 1$, we have $a_i(n)\dto \Geo(\gb^d)$, again with convergence of all moments.
 \end{enumerate}
\end{theorem}

\begin{remark}
 We note that the critical window $\gT\big(n^{(d-1)/d}\big)$ from Theorem~\ref{thm:limit}(\ref{it:prop2}) is consistent with the results in \cite{Jan2023}, see especially Section~9.1 therein. The limiting random variable $\Xi_\ga$ has the same distribution as $1-Y$ conditioned on $1-Y>0$, where $Y$ is Fr\'echet distributed with shape $d-1$ and scale $s=\big(\ga^d/(d-1)\big)^{1/(d-1)}$, that is, $Y$ has cumulative distribution function 
 \begin{equation}
  \IP(Y<x) = e^{-(x/s)^{-(d-1)}} \qquad \text{ for } \qquad x>0.
 \end{equation}
\end{remark}

To formulate our second main result, we adapt the notion of a threshold from combinatorial probability and random graph theory (see e.g. \cite[Chapter~1.5]{JLR00}). Specifically, we say that a function $g(n)$ is an \emph{threshold} for a set $S$ of vertices in $G_n$ if
\begin{equation}\label{eq:defthreshold}
 \IP(S\cap [m]\neq \emptyset) \to \begin{cases}
                                   0 & \text{ if } m=m(n)=o(g(n))\\
                                   1 & \text{ if } m=m(n)=\go(g(n)).
                                  \end{cases}
\end{equation}
In the context of the \ER random graph $G(n,m)$, this corresponds to a coarse threshold. We note that $g(n)$ is not uniquely determined, since any $\tilde g(n)$ with $\tilde g(n)=\gT(g(n))$ also satisfies \eqref{eq:defthreshold}. We note further that the threshold gives the position of the earliest vertex belonging to $S$.

This notion enables a tidy formulation for our answer to problem~(\ref{item:q2}):
\begin{theorem}\label{thm:threshold}
 The threshold for $[n]\setminus D_n$ in $G_n$ is $n^{\frac{d-1}{d+1}}$.
\end{theorem}

\begin{remark}
 The proof of Theorem~\ref{thm:threshold} shows specifically that for every $\gep>0$, there exist constants $0<c,C<\infty$ such that
 \begin{equation}\label{eq:proofstratthreshold}
  \IP\left( \left[\floor{cn^{(d-1)/(d+1)}}\right] \setminus D_n \neq \emptyset \right) < \gep
  \qquad \text{ and } \qquad
  \IP\left( \left[\floor{Cn^{(d-1)/(d+1)}}\right] \setminus D_n \neq \emptyset \right) > 1-\gep.
 \end{equation}
 For the first statement (the lower bound on the threshold), we employ the first moment method in combination with elementary properties of the ancestor process. For the upper bound on the threshold, we use a coupling trick to show that with probability close to 1, there is a vertex $\leq Cn^{(d-1)/(d+1)}$ that gains no extra ancestors until the range $\gT\big(n^{(d-1)/d}\big)$. This means that the further evolution of such a vertex's ancestry behaves just like a vertex from this range, which has a non-trivial limiting probability to be in $D_n$.
\end{remark}

\begin{remark}
 Due to the inequalities in \eqref{eq:proofstratthreshold}, it may still be the case that the earliest non-descendant is concentrated in a narrower regime than $\gT\big( n^{(d-1)/(d+1)} \big)$. We suspect however that this is not the case.
\end{remark}

Among the ingredients for the proof of Theorem~\ref{thm:threshold} is the following result, which may be of independent interest given that the number of source nodes has been investigated in \cite{KS2017,MPSU2014,TM2001}:

\begin{theorem}\label{thm:sources}
 For $M>0$, denote by $Z_M$ the number of sources (that is, vertices with in-degree 0 in $G_n$) among the first $\floor{Mn^{d/(d+1)}}$ vertices. As $n\to\infty$, we have
 \begin{equation}\label{eq:sourcePoi}
  Z_M\dto \Poi\left(\frac{M^{d+1}}{d+1}\right),
 \end{equation}
 with convergence of all moments. In particular, the threshold for source nodes in $G_n$ is $n^{d/(d+1)}$.
\end{theorem}

\begin{remark}\label{rem:PPP}
 A stronger statement is in fact true: The point processes $Z_n$ defined on $[0,\infty)$ via 
 \begin{equation}\label{eq:sourcepp}
  Z_n([a,b]) := \left|\left\{i: \floor{an^{d/(d+1)}} \leq i \leq \floor{bn^{d/(d+1)}} \text{ and $i$ is a source in $G_n$} \right\}\right|
 \end{equation}
 converges in distribution to an inhomogeneous Poisson point process on $[0,\infty)$ with intensity measure $x^d \Di x$.
\end{remark}

\subsection{Notation}

We will use $C$ for generic positive constants used in estimates, and their exact values might change from one occurrence to the next. Further, for some estimates we tacitly assume that $n$ is chosen sufficiently large. In the proofs below, all limits are taken as $n\to\infty$, unless otherwise indicated.

We say that an event holds \emph{with high probability} or \emph{whp} if it holds with probability $1-o(1)$. We write $\Geo(p)$ for the geometric distribution supported on the positive integers with parameter $p$, $\Poi(\gl)$ for the Poisson distribution with parameter $\gl$, and $\Exp(1)$ for the exponential distribution with parameter 1.

We denote the Beta and Gamma functions by $\Beta(a,b)$ and $\gC(z)$, respectively.

Throughout the paper, we use standard Landau notation: For two functions $f(n), g(n)>0$, we write $f=o(g)$, $f=O(g)$, $f=\gT(g)$, and $f=\go(g)$ as $n\to\infty$ if $\lim f/g=0$, $\limsup f/g<\infty$, $0<\liminf f/g\leq \limsup f/g<\infty$, and $\lim f/g=\infty$, respectively. We write $f\sim g$ to express $\lim f/g=1$.

\section{The Ancestry Process}\label{sec:ancestry}

Let $a_i(n):=|A_i(n)|$ for $i\in[n]$ denote the \emph{ancestry process} of the vertex $i$ in $G_n$.
Clearly,
\begin{equation}\label{eq:aintrivial}
a_i(i)=1\qquad \text{\ and\ }\qquad a_1(n)=n.
\end{equation}
Moreover, we have
\begin{equation}\label{eq:ainincrement}
 0\leq a_i(n)-a_i(n-1) \leq 1
\end{equation}
and note that $\{a_i(n)\}_{n\geq i}$ is a Markov chain on the positive integers, with transition probabilities given by
\begin{equation}\label{eq:aintransitionp}
 \IP(a_i(n)=j\mid a_i(n-1)=j) = \bigg( \frac{n-1-j}{n-1}\bigg)^d
\end{equation}
and, in view of \eqref{eq:ainincrement},
\begin{equation}\label{eq:aintransitionp2}
 \IP(a_i(n)=j+1\mid a_i(n-1)=j) = 1-\bigg( \frac{n-1-j}{n-1}\bigg)^d.
\end{equation}
Indeed, the event $a_i(n)=a_i(n-1)$ is exactly the event where all $d$ outgoing arrows from vertex $n$ point to $[n-1]\setminus A_i(n-1)$, which yields \eqref{eq:aintransitionp}.

\begin{remark}\label{rem:urn}
 The Markov chain $\{a_i(n)\}_{n\geq i}$ can also be thought of as an urn process. The urn is initiated at time $i$ to contain $i-1$ balls of colour $\emptyset$ and one ball of colour $\{i\}$ (the seemingly weird choice of ``colours'' will turn out to be useful). At time $n\geq i$, $d$ balls are drawn with replacement from the urn (independently of everything else), resulting in balls with colours $c_1,\dots, c_d$. Afterwards, an additional ball with colour $\bigcup_{s=1}^d c_s$ is added to the urn. We remark that similar urn processes have been used in \cite{KS2017} and \cite{MPSU2014} for the investigation of $G_n$. General multi-draw urn models were studied in a number of papers, for example \cite{KM2017} or \cite{LMS2018}. More recently, a deeper connection between the random recursive $d$-DAG and multi-draw P\'olya urns was formulated in \cite{MS2025}. However, in this paper the relevant phenomena occur when $i$ (and thus the initial configuration of the urn) depends on $n$, which is not covered by usual urn models.
\end{remark}

We have the following lemma for the distribution of $a_i(n)$:

\begin{lemma}\label{lemma:aindistr}
 For $2\leq i\leq n$ and $j\in[n+1-i]$, we have
 \begin{equation}\label{eq:aindistr}
  \IP(a_i(n)=j) = \bigg(\frac{i-1}{n-j}\bigg)^d \cdot \prod_{k=1}^{j-1} \left[1-\bigg(\frac{i-1}{n-k}\bigg)^d\,\right].
 \end{equation}
\end{lemma}
\begin{proof}
 Fixing $i\geq 2$, we proceed by induction on $n$. The base case is covered by equation~\eqref{eq:aintrivial}. In the induction step, we observe using \eqref{eq:ainincrement} and \eqref{eq:aintransitionp} for $j=1$ that
 \[
  \IP(a_i(n)=1) =  \bigg( \frac{n-2}{n-1}\bigg)^d \bigg(\frac{i-1}{n-2}\bigg)^d = \bigg( \frac{i-1}{n-1}\bigg)^d.
 \]
 Similarly, using \eqref{eq:ainincrement}--\eqref{eq:aintransitionp2} for $1\leq j\leq n-1$, we obtain:
 \begin{align*}
  \IP(a_i(n)=j+1) &= \bigg( \frac{n-2-j}{n-1}\bigg)^d \cdot \bigg(\frac{i-1}{n-2-j}\bigg)^d \cdot \prod_{k=1}^{j} \left[1-\bigg(\frac{i-1}{n-1-k}\bigg)^d\,\right]\\
  &\qquad + \left[1-\bigg( \frac{n-1-j}{n-1}\bigg)^d\,\right] \cdot \bigg(\frac{i-1}{n-1-j}\bigg)^d \cdot \prod_{k=1}^{j-1} \left[1-\bigg(\frac{i-1}{n-1-k}\bigg)^d\,\right]\\
  &= \bigg(\frac{i-1}{n-1-j}\bigg)^d \left[1-\bigg(\frac{i-1}{n-1}\bigg)^d\,\right] \prod_{k=1}^{j-1}\left[1-\bigg(\frac{i-1}{n-1-k}\bigg)^d\,\right],
 \end{align*}
 which simplifies to the right-hand side of \eqref{eq:aindistr} for $j+1$.
\end{proof}

\begin{remark}\label{rem:stopping}
 We can reinterpret equation~\eqref{eq:aindistr} in terms of a stopping time: Consider a sequence $\big(B_{n,i}(j)\big)_{j\geq 1}$ of independent Bernoulli trials, such that
 \begin{equation}
  B_{n,i}(j)\sim \Ber(p_{n,i}(j)) \qquad \text{ with } \qquad p_{n,i}(j) = \left(\frac{i-1}{n-j}\right)^d.
 \end{equation}
 Then the expression in equation~\eqref{eq:aindistr} equals $\IP(T=j)$ for $T=\min\{k\geq 1:B_{n,i}(k)=1\}$, and in particular $T=a_i(n)$ in distribution.
\end{remark}

\begin{lemma}\label{lemma:firstmoment}
 For all $n\geq 2$, we have the following:
 \begin{enumerate}[i.]
  \item\label{item:firstmomenti} The expected number of descendants of $n$ in $G_n$ is given by
  \begin{equation}\label{eq:EDn}
   \IE|D_n| = 2 + \sum_{i=2}^{n-1} \sum_{j=1}^{n-i} \left[1-\bigg(\frac{n-1-j}{n-1}\bigg)^d\,\right]
   \bigg(\frac{i-1}{n-1-j}\bigg)^d \cdot \prod_{k=1}^{j-1} \left[1-\bigg(\frac{i-1}{n-1-k}\bigg)^d\,\right],
  \end{equation}
  or, equivalently,
  \begin{equation}\label{eq:EDn1}
   \IE|D_{n+1}| = 2 + \sum_{i=1}^n \sum_{j=1}^{n-i} \left[1-\bigg(\frac{n-j}{n}\bigg)^d\right] \bigg(\frac{i}{n-j}\bigg)^d \cdot \prod_{k=1}^{j-1} \left[1-\bigg(\frac{i}{n-k}\bigg)^d\right].
  \end{equation}
  \item\label{item:firstmomentii} The expected number of joint descendants of $n+1$ and $n+2$ in $G_{n+2}$ is given by
  \begin{multline}\label{eq:EDncap}
   \IE|D_{n+1}\cap D_{n+2}| = 2 - \bigg(\frac{n}{n+1}\bigg)^d\\
   + \sum_{i=1}^n \sum_{j=1}^{n-i} \left[1-\bigg(\frac{n-j}{n+1}\bigg)^d\right]\cdot \left[1-\bigg(\frac{n-j}{n}\bigg)^d\right] \bigg(\frac{i}{n-j}\bigg)^d \cdot \prod_{k=1}^{j-1} \left[1-\bigg(\frac{i}{n-k}\bigg)^d\right]
  \end{multline}
 \end{enumerate}
\end{lemma}

\begin{proof}
 We write $X_i=\one\{i\in D_n\}$, and rely on observation~\eqref{eq:keyobs}. Thus
 \begin{align*}
  \IE|D_n|
  &= \sum_{i=1}^n \IP(X_i=1)
   = 2 + \sum_{i=2}^{n-1} \IP\big(a_i(n)=1+a_i(n-1)\big)\\
  &= 2 + \sum_{i=2}^{n-1} \sum_{j=1}^{n-i} \IP\big(a_i(n)=j+1\mid a_i(n-1)=j\big) \IP(a_i(n-1)=j)
 \end{align*}
 by the law of total probability. Equations~\eqref{eq:aintransitionp2} and \eqref{eq:aindistr} yield \eqref{eq:EDn}, and \eqref{eq:EDn1} follows after replacing $n$ by $n+1$, shifting $i$ to start the summation at 1, and then noting that $i=n$ leads to an empty inner sum.

 For the second part, write instead $X_i=\one\{i\in D_{n+1}\cap D_{n+2}\}$, where similarly $i\in D_{n+1}\cap D_{n+2}$ if and only if $a_i(n+2)=1+a_i(n+1)=2+a_i(n)$ (the exceptional cases being $X_1=1$ and $X_{n+1}=\one\{n+1\in D_{n+2}\}$, leading to the terms in before the double sum). We can hence, using the Markov property, write
 \begin{multline}
  \IP(X_i=1)
  = \sum_{j=1}^{n+1-i} \IP\big(a_i(n+2)=j+2\mid a_i(n+1)=j+1\big)\\
  \cdot \IP\big(a_i(n+1)=j+1 \mid a_i(n)=j\big) \IP\big(a_i(n)=j\big)
 \end{multline}
 and obtain \eqref{eq:EDncap} as before.
\end{proof}

\begin{remark}\label{rem:highermoments}
 A necessary ingredient to push our approach to Theorem~\ref{thm:jointdesc} to higher moments (and to convergence in distribution via the method of moments) is to determine $\IP(I\ssq D_n)$ for a fixed finite subset $I\ssq \IN$.

 We can extend the necessary constructions from Section~\ref{sec:ancestry}: Fix a finite set $I\ssq \IN$ and consider, for any subset $J\ssq I$, the set of joint ancestors
 \begin{equation}\label{eq:defAIJ}
  A_{I,J}(n) := \bigcap_{i\in J} A_i(n) \cap \bigcap_{i\in I\setminus J} [n]\setminus A_i(n).
 \end{equation}
 We define $a_{I,J}(n) := |A_{I,J}(n)|$ and set $a_I(n):=(a_{I,J}(n))_{J\ssq I}$. In other words, $a_I(n)$ is a $2^{|I|}$-dimensional vector indexed by subsets of $I$ whose $J$-coordinate records the number of vertices in $[n]$ that are ancestors to $J$ but not to $I\setminus J$ (the above $a_i(n)$ corresponds to $a_{\{i\},\{i\}}(n)$). Let $e_J$ denote the standard basis vector having 1 in the $J$-coordinate and 0 everywhere else.

 The process $\{a_I(n)\}_{n\geq 1}$ is again a Markov process taking values in $\IZ_{\geq 0}^{2^{|I|}}$, with $a_I(1)=e_{\{1\}}$ if $1\in I$ and $a_I(1)=e_{\emptyset}$ otherwise. In lieu of \eqref{eq:ainincrement}, we now have  $a_I(n)-a_I(n-1)=e_J$ with $J=\{i\in I: n\descend i\}$. However, there appears to be little hope of obtaining a similarly nice formula as \eqref{eq:aindistr} for the distribution of $a_I(n)$.

 Once again, this process can be thought of as a urn process, where the colours are given as subsets $J\ssq I$. At time $n$, we draw $d$ balls of colours $c_1,\dots,c_d$ with replacement from the urn, and add an additional ball of colour $\bigcup_{s=1}^d c_i$ if $n\notin I$, or instead of colour $\{n\} \cup \bigcup_{s=1}^d c_i$ if $n\in I$, to the urn. However, it is unclear if the urn interpretation is helpful for the higher moments, for the same reason as the one given in Remark~\ref{rem:urn}.
\end{remark}

\section{A Technical Theorem}\label{sec:technical}

We have seen in Lemma~\ref{lemma:firstmoment} that both the expressions for $\IE|D_{n+1}|$ and for $\IE|D_{n+1}\cap D_{n+2}|$ contain sums of the form
\begin{equation}
 \sum_{i=1}^n \sum_{j=1}^{n-i} q_n(i,j)\cdot \prod_{k=1}^{j-1} \left[1-\bigg(\frac{i}{n-k}\bigg)^d\,\right],
\end{equation}
where $q_n(i,j)$ is such that upon reparametrizing $\ga=in^{-(d-1)/d}$ and $\gb=j/n$, we obtain
\begin{equation}
 q_n(i,j) \sim \frac{1}{n} \hat q (\ga,\gb)
\end{equation}
for some function $\hat q:[0,\infty)\times [0,1]\to[0,\infty)$ that is no longer dependent on $n$. We will thus formulate and prove the following theorem to deal with such expressions:

\begin{theorem}\label{thm:technical}
 With notation as before, suppose there exists a constant $\gk>0$ such that
 \begin{equation}\label{eq:qnijest}
  0\leq q_n(i,j) \leq \gk \frac{j}{n^{d+1}} \frac{i^d}{(1-j/n)^d}
 \end{equation}
 for all $n\geq 1$, $i\in[n]$, and $j\in[n-i]$. Moreover assume that $n q_n(i,j)\to \hat q(\ga,\gb)$ for all $i$ and $j$ such that $in^{-(d-1)/d}\to \ga >0$ and $j/n\to \gb \in(0,1)$. Then
 \begin{multline}\label{eq:thmtech}
  n^{-\frac{d-1}{d}} \sum_{i=1}^n \sum_{j=1}^{n-i} q_n(i,j) \prod_{k=1}^{j-1} \left[1-\bigg(\frac{i}{n-k}\bigg)^d\,\right]\\
  \to \int_0^\infty \int_0^1 \hat q(\ga,\gb) e^{-\frac{\ga^d}{(d-1)}\big((1-\gb)^{1-d}-1\big)} \Di\gb \Di\ga
 \end{multline}
 as $n\to\infty$.
\end{theorem}

For the proof, we use a first order approximation to the summand on the left-hand side of \eqref{eq:thmtech}, together with an application of the dominated convergence theorem. It will be convenient to replace $i$ and $j$ by
\begin{equation}
 i(\ga)=i_n(\ga):= \floor{\ga n^{\frac{d-1}{d}}} \qquad \text{ and } \qquad j(\gb) = j_n(\gb) := \floor{\gb n}
\end{equation}
for $\ga>0, 0<\gb<1$. As it turns out, the case where $i(\ga)=1$ leads to complications with the $\gb=1$ case in later estimates, so we will first note that we can drop these terms because
\begin{equation}
 \sum_{j=1}^{n-1} q_n(1,j) \prod_{k=1}^{j-1} \left[1-\bigg(\frac{1}{n-k}\bigg)^d\,\right]
 \leq \gk \sum_{j=1}^{n-1} \frac{j}{n(n-j)^d} 
 = \gk \sum_{r=1}^{n-1} \frac{n-r}{n} \cdot \frac{1}{r^d}
 \leq \gk \zeta(d) < \infty
\end{equation}
by \eqref{eq:qnijest}. 

We further write
\begin{equation}
 f_n(\ga,\gb) := q_n(i(\ga),j(\gb)) \prod_{k=1}^{j(\gb)-1} \left[1-\bigg(\frac{i(\ga)}{n-k}\bigg)^d\,\right],
\end{equation}
so that we can express the left-hand side of \eqref{eq:thmtech} as
\begin{multline}\label{eq:sumasint}
 n^{-\frac{d-1}{d}} \sum_{i=2}^n \sum_{j=1}^{n-i} q_n(i,j) \prod_{k=1}^{j-1} \left[1-\bigg(\frac{i}{n-k}\bigg)^d\,\right]
 = \int_{2n^{-\frac{d-1}{d}}}^{n^{1/d}} \int_{\frac{1}{n}}^{1-\frac{i(\ga)}{n}+\frac{1}{n}} n f_n(\ga,\gb)\Di\gb \Di\ga\\
 = \int_0^\infty \int_0^1 \one\left\{(\ga,\gb)\in M_n\right\} nq_n(i(\ga),j(\gb))  \prod_{k=1}^{j(\gb)-1} \left[1-\bigg(\frac{i(\ga)}{n-k}\bigg)^d\,\right] \Di\gb \Di\ga
\end{multline}
where
\begin{equation}
 M_n=\left\{(\ga,\gb)\in \IR_{\geq 0}\times [0,1]: 2n^{-\frac{d-1}{d}} \leq\ga\leq n^{\frac{1}{d}}, \frac{1}{n} \leq \gb \leq 1-\frac{i(\ga)}{n}+\frac{1}{n} \right\}.
\end{equation}
Thus, it remains to show that
\begin{enumerate}[(a)]
 \item\label{item:goal1} the integrands on the right-hand side of \eqref{eq:sumasint} converge pointwise for all $\ga,\gb$ to the integrand in \eqref{eq:thmtech}; and
 \item\label{item:goal2} the integrands are dominated by an integrable function.
\end{enumerate}
Once these claims are established, the assertion of Theorem~\ref{thm:technical} follows from applying the dominated convergence theorem to equation~\eqref{eq:sumasint}.

\subsection{Proof of Claim~(\ref{item:goal1})}
 Fix $\ga>0, 0<\gb<1$. Then, for $n$ sufficiently large, $2n^{-(d-1)/d}\leq \ga \leq n^{1/d}$, and furthermore, since $i(\ga)/n = \ga n^{-1/d} + o(1)$, we also have $1/n\leq \gb \leq 1-(i(\ga)-1)/n$. Hence, $\one\left\{(\ga,\gb)\in M_n\right\} \to 1$ pointwise as $n\to\infty$. Moreover, by the assumption of the theorem, we have $nq_n(i(\ga),j(\gb))\to \hat q(\ga,\gb)$. For the product, we use the Taylor expansion $\log(1-x)=-x-x^2/2-...$ to obtain
 \begin{equation}\label{eq:prodexp}
  \prod_{k=1}^{j(\gb)-1} \left[1-\bigg(\frac{i(\ga)}{n-k}\bigg)^d\,\right]
  = \exp\left( - \sum_{k=1}^{j(\gb)-1} \bigg(\frac{i(\ga)}{n-k}\bigg)^d - O\Bigg(\sum_{k=1}^{j(\gb)-1} \bigg(\frac{i(\ga)}{n-k}\bigg)^{2d} \Bigg)\right).
 \end{equation}
 Let us consider the sum occurring here. We have
 \begin{align}
  \sum_{k=1}^{j(\gb)-1} \left(\frac{i(\ga)}{n-k}\right)^d
  &= \frac{\floor{\ga n^{(d-1)/d}}^d}{n^{d-1}} \sum_{k=1}^{j(\gb)-1} \frac{1}{n} \frac{1}{(1-k/n)^d}\\
  &= (1+o(1)) \ga^d \int_{1/n}^{\floor{\gb n}/n} (1-x)^{-d} \Di x\\
  &= (1+o(1)) \frac{\ga^d}{d-1} \big((1-\gb)^{1-d}-1\big)
 \end{align}
 by noticing that the sum in the second expression is a Riemann sum for the integral. Repeating this exercise with $2d$ in place of $d$ yields
 \begin{equation}
  \sum_{k=1}^{j(\gb)-1} \left(\frac{i(\ga)}{n-k}\right)^{2d}
  = (1+o(1)) \frac{\ga^{2d}}{n(2d-1)} \big((1-\gb)^{1-2d}-1\big)
  = O(1/n),
 \end{equation}
 so the error term in equation~\eqref{eq:prodexp} is $o(1)$ and we obtain
 \begin{equation}
  \prod_{k=1}^{j(\gb)-1} \left[1-\bigg(\frac{i(\ga)}{n-k}\bigg)^d\,\right] \to
  \exp\left( -\frac{\ga^d}{d-1} \big((1-\gb)^{1-d}-1\big)\right).
 \end{equation}
 Thus, the integrand in \eqref{eq:sumasint} converges pointwise to the one in \eqref{eq:thmtech}.

\subsection{Proof of Claim~(\ref{item:goal2})}
 We first note that, by the assumption on $q_n(i,j)$ in Theorem~\ref{thm:technical}, we have
 \begin{equation}
  nq_n(i(\ga),j(\gb))
  \leq \gk \frac{j(\gb)}{n^d} \frac{i(\ga)^d}{(1-j(\gb)/n)^d}
  \leq C\ga^d \frac{\gb}{(1-\gb)^d}
 \end{equation}
 for all $\ga>0, 0<\gb<1$. Furthermore, we can upper bound the product in \eqref{eq:sumasint} by dropping the $O$-term in equation~\eqref{eq:prodexp}, hence
 \begin{equation}\label{eq:prodest}
  \prod_{k=1}^{j(\gb)-1} \left[1-\bigg(\frac{i(\ga)}{n-k}\bigg)^d\,\right]
  \leq \exp\Bigg( - \sum_{k=1}^{j(\gb)-1} \bigg(\frac{i(\ga)}{n-k}\bigg)^d \Bigg).
 \end{equation}
 Observing that $x\mapsto (n-x)^{-d}$ is a monotonically increasing function for $x<n$, we can derive lower bounds to this sum:
 \begin{align}
  \sum_{k=1}^{j(\gb)-1} \bigg(\frac{i(\ga)}{n-k}\bigg)^d
  &\geq \int_0^{j(\gb)-1} \frac{i(\ga)^d}{(n-x)^d} \Di x
   = \frac{i(\ga)^d}{d-1} \big((n-j(\gb)+1)^{1-d} - n^{1-d}\big)\\
  &\geq \frac{(1+o(1))\ga^d}{d-1} \big((1-\gb)^{1-d} -1\big).
 \end{align}
 In particular, for $\gb\geq 1/2$ (or any other positive constant), we have the corresponding lower bound
 \begin{equation}\label{eq:sumlower}
  \sum_{k=1}^{j(\gb)-1} \bigg(\frac{i(\ga)}{n-k}\bigg)^d \geq C\ga^d (1-\gb)^{1-d}
 \end{equation}
 If instead $\gb$ is close to 0, say $\gb<1/2$, we obtain the elementary bound
 \begin{equation}\label{eq:sumlower2}
  \sum_{k=1}^{j(\gb)-1} \bigg(\frac{i(\ga)}{n-k}\bigg)^d
  \geq (j(\gb)-1) \frac{i(\ga)^d}{n^d}
  \geq C \ga^d \gb.
 \end{equation}

 Combining the bounds \eqref{eq:sumlower} and \eqref{eq:sumlower2} with \eqref{eq:prodest}, we thus arrive at
 \begin{equation}
  nq_n(i(\ga),j(\gb))  \prod_{k=1}^{j(\gb)-1} \left[1-\bigg(\frac{i(\ga)}{n-k}\bigg)^d\,\right]
  \leq C\ga^d \frac{\gb}{(1-\gb)^d} e^{-C\ga^d \gb(1-\gb)^{1-d}}
 \end{equation}
 for $(\ga,\gb)\in\IR_{>0}\times(0,1)$. It is welcome news that this bound is integrable; indeed, by Toninelli's theorem we may first integrate over $\ga$, and by setting $k_\gb=C\gb(1-\gb)^{1-d}$ constant we obtain
 \begin{equation}\label{eq:innerint}
  \int_0^\infty \ga^d e^{-\ga^d k_{\gb}} \Di\ga
  = \int_0^\infty \frac{1}{d k_{\gb}^{1+1/d}} u^{1/d}e^{-u} \Di u
  = \frac{1}{d}\,\gC\!\left(\frac{d+1}{d}\right) k_{\gb}^{-\frac{d+1}{d}},
 \end{equation}
 via the substitution $u=k_{\gb} \ga^d$. Hence
 \begin{align}
  \int_0^\infty \int_0^1 C\ga^d \frac{\gb}{(1-\gb)^d} e^{-K\ga^d \gb(1-\gb)^{1-d}} \Di\gb\Di\ga
  &= C \int_0^1 \gb(1-\gb)^{-d} \big(\gb(1-\gb)^{1-d}\big)^{-\frac{d+1}{d}} \Di\gb\\
  &= C \Beta\!\left(\frac{d-1}{d},\frac{d-1}{d}\right) < \infty,
 \end{align}
 as desired.

\section{Evaluation of Integrals}\label{sec:integrals}

For the proof of Theorem~\ref{thm:jointdesc}, it now only remains to evaluate the integral obtained from applying Theorem~\ref{thm:technical} to the expression in Lemma~\ref{lemma:firstmoment}\ref{item:firstmomentii}. However, in the interest of presentation we will first re-derive the first moment convergence in Theorem~\ref{thm:Janson}, as the computations are similar but easier.

\subsection{Asymptotics for \texorpdfstring{$\IE|D_n|$}{E|Dn|}}\label{ssec:EDn}

Recall the formula \eqref{eq:EDn} for $\IE|D_n|$. In the framework of Theorem~\ref{thm:technical}, this corresponds to a kernel
\begin{equation}
 q_n(i,j) = \left[1-\bigg(\frac{n-j}{n}\bigg)^d\,\right]\bigg(\frac{i}{n-j}\bigg)^d
\end{equation}
with corresponding
\begin{equation}
 \hat q_n(\ga,\gb) \to \hat q(\ga,\gb)= \big(1-(1-\gb)^d\big) \frac{\ga^d}{(1-\gb)^d}.
\end{equation}
Thus, applying Theorem~\ref{thm:technical} yields
\begin{equation}\label{eq:T1doubleint}
 n^{-\frac{d-1}{d}} \IE|D_n|
 \to \int_0^\infty \int_0^1 \frac{1-(1-\gb)^d}{(1-\gb)^d} \ga^d
     \exp\bigg(-\ga^d \frac{(1-\gb)^{1-d}-1}{d-1}\bigg) \Di\gb \Di\ga =: I
\end{equation}
and it only remains to evaulate the integral $I$. By Toninelli's theorem, we may exchange the order of  integration, and the resulting inner integral 
\begin{equation}
 \int_0^\infty \ga^d e^{-\ga^d k_\gb} \Di \gb
\end{equation}
is the same as in equation~\eqref{eq:innerint} with $k_{\gb}=\big((1-\gb)^{1-d}-1\big)/(d-1)$. Plugging the result from there into \eqref{eq:T1doubleint} reveals that
\begin{align}
 I
 &= \frac{1}{d} \gC\bigg(\frac{d+1}{d}\bigg)
  \int_0^1 \frac{1-(1-\gb)^d}{(1-\gb)^d} \bigg(\frac{d-1}{(1-\gb)^{1-d}-1}\bigg)^{1+1/d} \Di\gb\\
 &= \frac{(d-1)^{1/d}}{d} \gC\bigg(\frac{d+1}{d}\bigg)
  \int_0^1 \frac{1-v^{\frac{d}{d-1}}}{v^{\frac{d-1}{d}}(1-v)^{\frac{d+1}{d}}} \Di v
\end{align}
where we substituted $v=(1-\gb)^{d-1}$. It is tempting to split the integrand into a difference, but this only leads to a difference of two divergent beta-integrals. Instead, define for $\Re a>0, \Re b>-1$ the function
\[
 f_{a,b}(s) := \int_0^1 (1-v^a)v^b(1-v)^s \Di v.
\]
This integral converges absolutely for $\Re s>-2$ since for $v\to 1$ we have $|(1-v^a)v^b(1-v)^s|=(1-v)^{\Re s} O(1-v)$. Moreover, on a compact set $K\ssq\{\Re s>-2\}$, the integrand is dominated for the same reason by $C_K v^{\Re b}(1-v)^{1+r_K}$ where $r_K=\min_{z\in K} \Re z$ and $C_K>0$ a constant. Since the integrand is also an analytic function on $\{\Re s>-2\}$ for every fixed $v\in (0,1)$, it follows from differentiation under the integral sign that $f_{a,b}(s)$ is analytic on $\{\Re s>-2\}$.

Additionally, for $\Re s >-1$, we can write
\[
 f_{a,b}(s) = \int_0^1 v^b(1-v)^s \Di v - \int_0^1 v^{a+b} (1-v)^s\Di v = \Beta(b+1,s+1) - \Beta(a+b+1,s+1)
\]
and by analytic continuation, the resulting identity between $f_{a,b}$ and the Beta function remains valid on the larger domain $\{\Re s>-2\}$. Returning to $I$, we thus obtain
\begin{equation}\label{eq:Ialmostfinal}
 I= \frac{(d-1)^{1/d}}{d} \gC\bigg(\frac{d+1}{d}\bigg)
    \left[ \Beta\bigg(\frac{1}{d},-\frac{1}{d}\bigg) - \Beta\bigg(\frac{d}{d-1}+\frac{1}{d},-\frac{1}{d}\bigg)\right].
\end{equation}
Here, we use the analytic continuation of the Gamma function to $\IC\setminus\{0,-1,-2,\dots\}$ to evaluate $\Beta(z,-1/d) = \gC(z)\gC(-1/d)/\gC(z-1/d)$. For $z=1/d$, this is still meaningfully defined as 0 since $\gC(z)^{-1}$ is an entire function; the other beta function is evaluated to be
\begin{equation}
 \Beta\bigg(\frac{d}{d-1}+\frac{1}{d},-\frac{1}{d}\bigg)
 = \frac{\gC\left(\frac{d}{d-1}+\frac{1}{d}\right) \gC\left(-\frac{1}{d}\right)}{\gC\left(\frac{d}{d-1}\right)}
 = -\frac{\pi}{\sin(\pi/d)} \frac{\gC\left(\frac{d}{d-1}+\frac{1}{d}\right)}{\gC\left(\frac{d}{d-1}\right) \gC\left(\frac{d+1}{d}\right)},
\end{equation}
where we used the reflection identity $\gC(z)\gC(1-z)=\frac{\pi}{\sin(\pi z)}$ for $z=-1/d$. Plugging this into \eqref{eq:Ialmostfinal} simplifies to the expression in \eqref{eq:thmSvante}.

\subsection{Proof of Theorem~\ref{thm:jointdesc}}\label{ssec:Ejointdesc}

This proof now runs analogously to the derivation in Section~\ref{ssec:EDn}: Applying Theorem~\ref{thm:technical} to the expression obtained in Lemma~\ref{lemma:firstmoment}\ref{item:firstmomentii}, we obtain
\begin{equation}
 n^{-\frac{d-1}{d}} \IE|D_n\cap D_{n+1}|
 \to \int_0^\infty \int_0^1 \frac{\big(1-(1-\gb)^d\big)^2}{(1-\gb)^d} \ga^d
     \exp\bigg(-\ga^d \frac{(1-\gb)^{1-d}-1}{d-1}\bigg) \Di\gb \Di\ga =: I.
\end{equation}
Nothing changes in the application of Toninelli's theorem and the evaluation of the integral in \eqref{eq:innerint}, which leaves us with
\begin{align}
 I
 &= \frac{1}{d} \gC\bigg(\frac{d+1}{d}\bigg)
    \int_0^1 \frac{\big(1-(1-\gb)^d\big)^2}{(1-\gb)^d} \bigg(\frac{d-1}{(1-\gb)^{1-d}-1}\bigg)^{1+1/d} \Di\gb\\
 &= \frac{(d-1)^{1/d}}{d} \gC\bigg(\frac{d+1}{d}\bigg)
    \int_0^1 \frac{\left(1-v^{\frac{d}{d-1}}\right)^2}{v^{\frac{d-1}{d}}(1-v)^{\frac{d+1}{d}}} \Di v.
\end{align}
We now define the slightly modified auxiliary function
\begin{equation}
 f_{a,b}(s)=\int_0^1 (1-v^a)^2 v^b (1-v)^s \Di s
\end{equation}
for $\Re a>0, \Re b>-1$. Repeating the argument from Section~\ref{ssec:EDn} reveals that $f_{a,b}$ is analytic on $\{\Re s>-3\}$, and once again we obtain through analytic continuation (after expanding the binomial term) that
\begin{equation}
 I= \frac{(d-1)^{1/d}}{d} \gC\bigg(\frac{d+1}{d}\bigg) \left[\Beta\bigg(\frac{1}{d},-\frac{1}{d}\bigg) - 2\Beta\bigg(\frac{d}{d-1}+\frac{1}{d},-\frac{1}{d}\bigg) + \Beta\bigg(\frac{2d}{d-1}+\frac{1}{d},-\frac{1}{d}\bigg)\right]
\end{equation}
which can be evaluated as above. We obtain \eqref{eq:thmjointdesc}.

\section{Proof of Theorem~\ref{thm:limit}}\label{sec:Prooflimit}

Throughout the proof, we rely on the stopping time $T$ from Remark~\ref{rem:stopping}, as it makes it easier to write down exact expressions for a number of probabilities.

Beginning with (\ref{it:prop1}), let $\gep>0$. By \eqref{eq:aindistr}, we have
\begin{align}
 \IP\left(\frac{T}{n}<1-\gep\right)
 &= \sum_{j=1}^{\floor{(1-\gep)n}} \left(\frac{i-1}{n-j}\right)^d \prod_{k=1}^{j-1} \left[1-\left(\frac{i-1}{n-k}\right)^d\right]\\
 &\leq \sum_{j=1}^{\floor{(1-\gep)n}} \left(\frac{i-1}{n-j}\right)^d
 \leq (1-\gep)n \left(\frac{i}{\gep n}\right)^d = o(1),
\end{align}
since $i=o\big(n^{(d-1)/d}\big)$. On the other hand, $a_i(n)/n\leq 1$ holds trivially, and the assertion follows.

Regarding part (\ref{it:prop2}), we emulate parts of the proof of Theorem~\ref{thm:technical} in Section~\ref{sec:technical}. To be precise, fix $\gc\in (0,1)$ and consider
\begin{align}
 \IP\big( T\leq \gc n \big)
 &=\sum_{j=1}^{\floor{\gc n}} \left(\frac{i-1}{n-j}\right)^d \prod_{k=1}^{j-1} \left[1-\left(\frac{i-1}{n-k}\right)^d\right]\\
 &= n \int_0^1 \one\left\{\frac{1}{n}\leq \gb\leq \frac{\floor{\gc n}+1}{n}\right\} \left(\frac{i-1}{n-j(\gb)}\right)^d \prod_{k=1}^{j(\gb)-1} \left[1-\left(\frac{i-1}{n-k}\right)^d\right] \Di\gb
\end{align}
Just as in the proof of Claim~(\ref{item:goal1}), in particular by recycling the estimates applied to equation~\eqref{eq:prodexp}, one can show that the integrand converges pointwise to
\begin{equation}
 \one\{0<\gb\leq \gc\} \ga^d (1-\gb)^{-d} e^{-\frac{\ga^d}{d-1}\big((1-\gb)^{1-d}-1\big)},
\end{equation}
and moreover from the proof of Claim~(\ref{item:goal2}) we can provide a majorant of the form
\begin{equation}
 C\ga^d (1-\gb)^{-d} e^{-C\ga^d\gb(1-\gb)^{1-d}}
\end{equation}
which is integrable even over the entire interval $[0,1]$ (the singularity at $\gb=1$ is compensated by the exponential decay of $e^{(1-\gb)^{1-d})}$). By dominated convergence this yields the convergence to a random variable $\Xi_\ga$ with density function $\ga^d (1-\gb)^{-d} e^{-\frac{\ga^d}{d-1}\big((1-\gb)^{1-d}-1\big)}$, which can be integrated to obtain the cumulative distribution function in \eqref{eq:thmlimitXi}. Moreover, since $T/n$ is bounded, we can conclude the moment convergence in (\ref{it:prop2}).

For the proof of (\ref{it:prop3}), fix $x>0$ and compute
\begin{align}\label{eq:dtoexp}
 \IP\big(T > x(n/i)^d \big)
 &= \prod_{k=1}^{\floor{x(n/i)^d}} \left[1-\left(\frac{i-1}{n-k}\right)^d\right]\notag \\
 &= \exp\left(-\sum_{k=1}^{\floor{x(n/i)^d}} \left(\frac{i-1}{n-k}\right)^d + O\Bigg( \sum_{k=1}^{\floor{x(n/i)^d}} \left(\frac{i-1}{n-k}\right)^{2d} \Bigg) \right).
\end{align}
Note that for $i=\go\big(n^{(d-1)/d}\big)$ with $k\leq x (n/i)^d$, we have $n-k=n(1-o(1))$ uniformly in $k$. Accordingly, we have
\begin{equation}
 \sum_{k=1}^{\floor{x(n/i)^d}} \left(\frac{i-1}{n-k}\right)^{2d}
 = (1+o(1)) x \left(\frac{n}{i}\right)^d \left(\frac{i}{n}\right)^{2d} =o(1)
\end{equation}
since $i=o(n)$, so the $O$-term vanishes, and similarly,
\begin{equation}
 \sum_{k=1}^{\floor{x(n/i)^d}} \left(\frac{i-1}{n-k}\right)^d
 = (1+o(1)) x.
\end{equation}
Plugging this into \eqref{eq:dtoexp} yields $\IP\big(T> x(n/i)^d\big)\to e^{-x}$ and establishes convergence in distribution. Convergence of the moments follows by a standard uniform integrability argument: We note that by replacing the $o(1)$-terms in the evalutaion of \eqref{eq:dtoexp} with explicit constants, we obtain $\IP\big(T> x(n/i)^d\big) < e^{-cx}$ for some $0<c<1$. We now choose $0<\gc<c$, so that
\begin{align}
 \IE \left[ e^{\gc (i/n)^d T}\right]
 &= \int_0^\infty \IP\left(e^{\gc (i/n)^d T} > u \right) \Di u
  = 1 + \int_0^\infty \gc e^{\gc v} \IP\big( (i/n)^d T > v\big)\Di v\\
 &\leq 1 + \gc \int_0^\infty e^{(\gc-c)v} \Di v.
\end{align}
Hence
\begin{equation}
 \sup_n \IE\left[ e^{\gc (i/n)^d T}\right] < \infty,
\end{equation}
which implies that the $r$-th moments of $(i/n)^d T$ are uniformly bounded for each $r\in \IN$. Thus $(i/n)^{rd} T^r$ is uniformly integrable for each $r$, and the moments converge.

For part (\ref{it:prop4}) we consider $i\sim \gb n$ and note that $p_{n,i}(j)\to \gb^d$ pointwise for  $j\in\IN$. Thus
\begin{equation}
 \IP(T=j) \to \gb^d (1-\gb^d)^{j-1}
\end{equation}
for all $j\in\IN$, which yields convergence in distribution. The moment convergence follows similarly to part (\ref{it:prop3}) after observing that for all $j$, eventually we have $\IP(T\geq j) \leq (1-(\gb/2)^d)^{j-1}$.

\section{Proof of Theorem~\ref{thm:threshold}}\label{sec:Proofthreshold}

\subsection{Proof of Theorem~\ref{thm:threshold}, lower bound}\label{ssec:Prooflower}
We establish the lower bound on the threshold in Theorem~\ref{thm:threshold} by showing the following about the expected number of non-descendants up to the threshold regime:
\begin{lemma}
 As $n\to\infty$, the expected number of non-descendants of $n$ in $G_n$ among the first $m$ vertices satisfies
 \begin{equation}\label{eq:nondesc}
  \IE \big|[m]\setminus D_n\big| \to \frac{1}{d+1}\ga^{d+1}
 \end{equation}
 for $m\sim \ga n^{\frac{d-1}{d+1}}$.
\end{lemma}

This lemma, combined with the first moment method, yields
\[
 \IP\big( [m]\setminus D_n \neq\emptyset \big) \leq \IE \big| [m]\setminus D_n \big| \to \frac{1}{d+1}\ga^{d+1}
\]
and in particular the first claim in \eqref{eq:proofstratthreshold}, after choosing $\ga>0$ sufficiently small.

\begin{proof}
 For $i,m\in[n]$, define $Y_i = \one\{i\notin D_n\}$ and $X_m = \big| [m]\setminus D_n \big|$, so that
 \begin{equation}
  \IE X_m = \IE \big| [m]\setminus D_n \big| = \sum_{i\leq m} \IE Y_i.
 \end{equation}
 Emulating the proof of Lemma~\ref{lemma:firstmoment}, we note that
 \begin{align*}
  \IE Y_i
  &= \IE \one\{a_i(n)=a_i(n-1)\}
   = \sum_{j=1}^{n-i} \IP\big(a_i(n)=j \mid a_i(n-1)=j \big) \IP(a_i(n-1)=j)\\
  &= \sum_{j=1}^{n-i} \left(\frac{n-1-j}{n-1}\right)^d \left(\frac{i-1}{n-1-j}\right)^d \cdot \prod_{k=1}^{j-1} \left[1-\left(\frac{i-1}{n-1-k}\right)^d\right],
 \end{align*}
 which yields (after replacing $n$ by $n+1$ and shifting $i$ by 1, noting that $\IE Y_1=0$)
 \begin{equation}\label{eq:EXm}
  \IE X_m = \sum_{i=1}^{m-1} \sum_{j=1}^{n-i} \left(\frac{i}{n}\right)^d \cdot \prod_{k=1}^{j-1} \left[1-\left(\frac{i}{n-k}\right)^d\right].
 \end{equation}
 Assume now $m\sim\ga n^{(d-1)/(d+1)}$ for $\ga > 0$. We first observe that the product converges to 1, since after reparametrizing $j=j(\gb)=\floor{\gb n}$ for $\gb\in(0,1)$, we can imitate equation~\eqref{eq:prodexp} and the estimates below it, to obtain
 \begin{multline}
  1
  \geq \prod_{k=1}^{j-1} \left[1-\left(\frac{i}{n-k}\right)^d\right]
  = \exp\left( -\frac{i^d}{n^d}\sum_{k=1}^{j(\gb)-1}\!\! \frac{1}{(1-k/n)^d} + O \Bigg( \frac{i^{2d}}{n^{2d}} \sum_{k=1}^{j(\gb)-1}\!\! \frac{1}{(1-k/n)^{2d}} \Bigg)  \right)\\
  = \exp\left( - O\left(n^{-\frac{d-1}{d+1}}\right) \frac{(1-\gb)^{1-d} -1}{d-1} + O\left( n^{-\frac{3d-1}{d+1}} \frac{(1-\gb)^{1-d}-1}{d-1}\right) \right)
  = \exp(-o(1)),
 \end{multline}
 where we used $i\leq m$. Hence, by comparison to an integral,
 \begin{align}
  \IE X_m
  &\leq \sum_{i=1}^m (n-i) \frac{i^d}{n^d}
  = n^{1-d} \sum_{i=1}^m i^d - n^{-d} \sum_{i=1}^m i^{d+1}\\
  &\sim \frac{1}{d+1}n^{1-d}m^{d+1} - \frac{1}{d+2} n^{-d}m^{d+2}
  \sim \frac{1}{d+1} \ga^{d+1}.
 \end{align}
 This achieves two things: For one, the majorant $(n-i)(i/n)^d$ to the $i$-summand in equation~\eqref{eq:EXm} is summable; but moreover we have seen that this majorant is also the pointwise limit of the summands. Thus, \eqref{eq:nondesc} follows by the dominated convergence theorem.
\end{proof}

\subsection{Earliest sources. Proof of Theorem~\ref{thm:sources}}\label{ssec:Proofsources}
In this section we determine the location of the earliest source in $G_n$, starting with a lemma concerning the joint probability of $r$ vertices to be sources.

\begin{lemma}\label{lemma:jointsources}
 Let $1\leq i_1 < \dots < i_r \leq n$ for some $r\in\IN$. Then
 \begin{equation}\label{eq:jointsources}
  \IP\big(a_{i_1}(n)=\dots=a_{i_r}(n)=1) = \prod_{j=1}^r \left(\frac{i_j-j}{n-j}\right)^d.
 \end{equation}
\end{lemma}
\begin{proof}
 For brevity, write $\mc E_{i_1,\dots,i_r}(n)$ for the event $\{a_{i_1}(n)=\dots=a_{i_r}(n)=1\}$.

 We use induction in $r$, with the base case being obtained from Lemma~\ref{lemma:aindistr}. Suppose \eqref{eq:jointsources} holds for some $r\geq 1$. Then, with $i_r< i_{r+1}\leq n$, we jointly condition the ancestor processes on the event that they are all $1$ at time $i_{r+1}$:
 \begin{equation}
  \IP\big( \mc E_{i_1,\dots,i_r,i_{r+1}}(n) \big)
  = \IP\big( \mc E_{i_1,\dots,i_r,i_{r+1}}(i_{r+1}) \big) \cdot \IP\big( \mc E_{i_1,\dots,i_r,i_{r+1}}(n) \mid \mc E_{i_1,\dots,i_r,i_{r+1}}(i_{r+1}) \big).
 \end{equation}
 Since $a_{i_{r+1}}(i_{r+1})=1$ by equation~\eqref{eq:aintrivial}, we have $\mc E_{i_1,\dots,i_r,i_{r+1}}(i_{r+1}) = \mc E_{i_1,\dots,i_r}(i_{r+1})$, and the induction hypothesis yields
 \begin{equation}\label{eq:jointsourceproof}
  \IP\big( \mc E_{i_1,\dots,i_r,i_{r+1}}(n) \big) = \IP\big( \mc E_{i_1,\dots,i_r,i_{r+1}}(n) \mid \mc E_{i_1,\dots,i_r,i_{r+1}}(i_{r+1}) \big) \cdot \prod_{j=1}^r \left(\frac{i_j-j}{i_{r+1}-j}\right)^d.
 \end{equation}
 For the conditional probability, note that every vertex $i=i_{r+1}+1,\dots,n$ in the $d$-DAG process needs to connect its $d$ outgoing arrow to any vertex that is not among $i_1,...,i_{r+1}$. This happens with probability
 \[
  \IP\big( \mc E_{i_1,\dots,i_r,i_{r+1}}(n) \mid \mc E_{i_1,\dots,i_r,i_{r+1}}(i_{r+1}) \big)
  = \prod_{k=i_{r+1}+1}^n \left( \frac{k-(r+1)-1}{k-1} \right)^d
  = \prod_{j=1}^{r+1} \left( \frac{i_{r+1}-j}{n-j} \right),
 \]
 after some telescoping. Plugging this into \eqref{eq:jointsourceproof}, we obtain the right-hand side of \eqref{eq:jointsources} for $r+1$, thus concluding the proof.
\end{proof}

\begin{proof}[Proof of Theorem~\ref{thm:sources}]
 We use the method of moments to establish Poisson convergence, showing that $\IE Z_M^{\underline{r}} := \IE\big[Z_M(Z_M-1)\cdots (Z_M-r+1)\big] \to M^{r(d+1)}/(d+1)^r$ for all $r\in\IN$ (for details on the method, see \cite[Chapter~6.1]{JLR00}). To this end, we introduce the indicator random variables
 \[
  Z(i) := \one\{i \text{ is a source in } G_n\} = \one\{a_i(n)=1\}
 \]
 for $i\in[n]$, and note that hence $Z_M=\sum_{i=1}^{\floor{Mn^{d/(d+1)}}} Z(i)$. By Lemma~\ref{lemma:aindistr}, we have $\IE Z(i)=\left(\frac{i-1}{n-1}\right)^d$. After setting $i(\gb):=\beta n^{d/(d+1)}$ and writing the sum
 \[
  \IE Z_M
  = \sum_{i=1}^{\floor{Mn^{d/(d+1)}}} \IE Z(i)
  = \sum_{i=1}^{\floor{Mn^{d/(d+1)}}} \left(\frac{i-1}{n-1}\right)^d
 \]
 as an integral over $\gb$ like in the proofs of Theorems~\ref{thm:technical} and \ref{thm:limit}, we obtain
 \begin{equation}\label{eq:EZM}
  \IE Z_M
  \sim n^{\frac{d}{d+1}} \int_0^M \Bigg(\frac{\gb n^{\frac{d}{d+1}}}{n-1}\Bigg)^d \Di\gb
  \sim \int_0^M \gb^d \Di\gb = \frac{M^{d+1}}{d+1}.
 \end{equation}

 For the higher factorial moments, we similarly write
 \begin{equation}
  \IE Z_M^{\underline r}
  = \sum_{\substack{i_1,\dots,i_r=1\\ \text{distinct}}}^{\floor{Mn^{d/(d+1)}}} \IE\Bigg[\prod_{j=1}^r Z(i_j)\Bigg]
  = \sum_{\substack{i_1,\dots,i_r=1\\ \text{distinct}}}^{\floor{Mn^{d/(d+1)}}} \IP\big( a_{i_1}(n)=\dots=a_{i_r}(n)=1 \big).
 \end{equation}
 With the help of Lemma~\ref{lemma:jointsources}, this can be written as
 \begin{equation}
  \IE Z_M^{\underline r} = \gS - \gS_{\tt{nad}}
  := \sum_{i_1,\dots,i_r} \prod_{j=1}^r \left( \frac{i_j-j}{n-j} \right)^d \ \ - \!\!\! \sum_{\substack{i_1,\dots,i_r\\ \text{ not all distinct}}} \prod_{j=1}^r \left( \frac{i_j-j}{n-j} \right)^d.
 \end{equation}
 The first sum factors neatly:
 \begin{equation}
  \gS = \prod_{j=1}^r \Bigg(\sum_{i_j=1}^{\floor{Mn^{d/d(+1)}}} \left(\frac{i_j-j}{n-j}\right)^d \Bigg) \sim \left(\frac{M^{d+1}}{d+1}\right)^r,
 \end{equation}
 where the asymptotic equivalence holds by (minor modifications of) \eqref{eq:EZM}. It thus remains to show that $\gS_{\tt{nad}}=o(1)$. To see this, we cover the set of $r$-tuples contributing to $\gS_{\tt{nad}}$ by the sets $S_{p,q}:=\{(i_1,\dots,i_r):i_p=i_q\}$ for $p,q\in[r], p < q$. On these sets, we have
 \begin{equation}
  \sum_{(i_1,\dots,i_r)\in S_{p,q}} \prod_{j=1}^r \left( \frac{i_j-j}{n-j} \right)^d
  = \Bigg( \sum_{i=1}^{\floor{Mn^{d/(d+1)}}} \left( \frac{(i-p)(i-q)}{(n-p)(n-q)} \right)^d \Bigg) \cdot
    \prod_{\substack{j=1\\ j\neq p,q}}^r \Bigg( \sum_{i_j=1}^{\floor{Mn^{d/(d+1)}}} \left(\frac{i_j-j}{n-j}\right)^d \Bigg).
 \end{equation}
 For the first factor, we note that
 \begin{equation}
  \sum_{i=1}^{\floor{Mn^{d/(d+1)}}} \left( \frac{(i-p)(i-q)}{(n-p)(n-q)} \right)^d
  \leq C \sum_{i=1}^{\floor{Mn^{d/(d+1)}}} \left( \frac{i}{n}\right)^{2d}
  \leq C n^{-2d} n^{\frac{(2d+1)d}{d+1}} = Cn^{-\frac{d}{d+1}} = o(1),
 \end{equation}
 and all remaining factors are $\sim \IE Z_M=O(1)$ by \eqref{eq:EZM}. Hence
 \begin{equation}
  \gS_{\tt{nad}} \leq \sum_{1\leq p<q\leq r} \sum_{(i_1,\dots,i_r)\in S_{p,q}} \prod_{j=1}^r \left( \frac{i_j-j}{n-j} \right)^d = o(1),
 \end{equation}
 and the convergence claim follows from the method of moments.

 The whp-statements now follow from considering $\IP(Z_M=0)\sim e^{-M^{d+1}/(d+1)}$ for values of $M$ that are arbitrarily small, respectively large.
\end{proof}

We briefly sketch how to extend the above proof to yield the convergence to a Poisson point process asserted in Remark~\ref{rem:PPP}. First note that $\IE Z_n([a,b])\sim \int_a^b \gb^d\Di \gb$, by a straightforward modification of equation~\eqref{eq:EZM}; all that remains to show is that the mixed factorial moments converge. To be precise, we have to show that for all $k\geq 1$, all $r_1,\dots,r_k\geq 1$, and all intervals $[a_1,b_1],\dots,[a_k,b_k]$ with $a_1\leq b_1<\dots<a_k\leq b_k$, we have 
\begin{equation}\label{eq:PPPproofsketch}
 \IE\left[ Z_n^{\underline{r_1}}([a_1,b_1]) \cdots Z_n^{\underline{r_k}}([a_k,b_k]) \right]
 \to \prod_{j=1}^k \left(\frac{b_j^{d+1}-a_j^{d+1}}{d+1}\right)^{r_j}.
\end{equation}
Denoting the limiting point process by $Z$, this shows that $\left(Z_n^{\underline{r_1}}([a_1,b_1]),\dots,Z_n^{\underline{r_k}}([a_k,b_k])\right)\dto \left(Z^{\underline{r_1}}([a_1,b_1]),\dots,Z^{\underline{r_k}}([a_k,b_k])\right)$ where the random variables in the limit are independent by \cite[Theorem~6.10]{JLR00}. The assertion then follows from \cite[Theorem~16.16]{Kal02}. 

We may write the left-hand side of \eqref{eq:PPPproofsketch} as a sum over tuples 
\begin{equation}
 (i_{1,1},\dots,i_{1,r_1}),\dots,(i_{k,1},\dots,i_{k,r_k})
\end{equation}
such that $\floor{a_j n^{d/(d+1)}}\leq i_{j,1},\dots,i_{j,r_j} \leq \floor{b_j n^{d/(d+1)}}$ for all $j=1,\dots,k$. Again the dominant part comes from the case when all entries are distinct, leading to the desired limit. The remaining case can again be covered by $O(1)$ sets $S_{j,p,q}$, consisting of those collections of tuples for which $i_{j,p}=i_{j,q}$; note that since the intervals $[a_j,b_j]$ are disjoint, identical indices can only occur within one value of $j$. The sum over $S_{j_0,p,q}$ will factor according to $j=1,\dots,k$, where by the previous computations the factors for $j\neq j_0$ will be of size $O(1)$. The remaining factor for $j=j_0$ contributes $o(1)$ to the product, so $S_{j_0,p,q}$ only contributes $o(1)$ in total, as in the proof of Theorem~\ref{thm:sources}.

\subsection{Limiting probability for decendancy}\label{ssec:Limprob}
We use this section to prove two statements about the probability that a vertex in $\gT\big( n^{d/(d+1)} \big)$ is a non-descendant of $n$.

Janson already observed that the descendants are concentrated in the range $O\big( n^{(d-1)/d} \big)$, where they have a non-trivial density in the limit, cf. \cite[Lemma~9.1]{Jan2023} for the case $d=2$. The following lemma complements this in a form that will be more useful for our purposes:
\begin{lemma}\label{lemma:notDn}
 As $n\to\infty$, we have for $M>0$ fixed
 \begin{equation}\label{eq:notDn}
  \IP\Big(\floor{Mn^{\frac{d-1}{d}}} \notin D_n \Big) \to M^d \int_0^1 e^{-\frac{M^d}{d-1}\big((1-\gb)^{1-d}-1\big)} \Di \gb.
 \end{equation}
\end{lemma}
\begin{proof}[Proof of Lemma~\ref{lemma:notDn}]
 We write $i(M):= \floor{Mn^{\frac{d-1}{d}}}$, and compute analogously to the proof of Lemma~\ref{lemma:firstmoment} that
 \begin{align*}
  \IP\big(i_M\notin D_n\big)
  &= \sum_{j=1}^{n-i(M)} \IP\big(a_{i(M)}(n)=j \mid a_{i(M)}(n-1)=j \big) \cdot \IP\big( a_{i(M)}(n-1)=j \big)\\
  &= \sum_{j=1}^{n-i(M)} \left(\frac{n-1-j}{n-1}\right)^d \left(\frac{i(M)-1}{n-1-j}\right)^d \cdot \prod_{k=1}^{j-1} \left[ 1-\left(\frac{i(M)-1}{n-1-k}\right)^d\right]\\
  &\sim  \left(\frac{i(M)}{n}\right)^d \sum_{j=1}^{n-i(M)} \prod_{k=1}^{j-1} \left[ 1-\left(\frac{i(M)-1}{n-1-k}\right)^d\right].
 \end{align*}
 Observing that $\big(i(M)/n\big)^d \sim \big(Mn^{-1/d}\big)^d = M^d/n$, we set $j=j(\gb)=\floor{\gb n}$ for $0<\gb<1$, and approximate the sum by a suitable integral as in previous proofs. We obtain
 \begin{equation}
  \IP\big(i_M\notin D_n\big)
  \sim M^d \int_0^1 \one\{ 1/n\leq \gb \leq 1-Mn^{-1/d}+1/n \}  \prod_{k=1}^{j(\gb)-1} \left[ 1-\left(\frac{i_M-1}{n-1-k}\right)^d\right] \Di\gb.
 \end{equation}
 Recycling the arguments from \eqref{eq:sumasint} and the following proofs of Claims~(\ref{item:goal1}) and (\ref{item:goal2}) shows that the integrand converges pointwise to the integrand on the right-hand side of \eqref{eq:notDn}; moreover the integrands are bounded by 1. Hence the dominated convergence theorem can be applied, and \eqref{eq:notDn} follows.
\end{proof}

\begin{lemma}\label{lemma:limnotDn}
 We have
 \begin{equation}\label{eq:limnotDn}
  \lim_{M\to\infty} M^d \int_0^1 e^{-\frac{M^d}{d-1}\big((1-\gb)^{1-d}-1\big)} \Di\gb =1.
 \end{equation}
\end{lemma}
\begin{proof}
 Upon substituting $u=M^d\gb$, we obtain
 \begin{equation}
  M^d \int_0^1 e^{-\frac{M^d}{d-1}\big((1-\gb)^{1-d}-1\big)} \Di\gb
  = \int_0^{M^d} e^{-\frac{M^d}{d-1}\left(\left(1-\frac{u}{M^d}\right)^{1-d} -1 \right)} \Di u.
 \end{equation}
 As $M\to\infty$, we obtain the following pointwise expansion for every $u>0$:
 \begin{equation}
  e^{-\frac{M^d}{d-1}\left(\left(1-\frac{u}{M^d}\right)^{1-d} -1 \right)}
  = e^{-\frac{M^d}{d-1}\left((d-1) \frac{u}{M^d} + O\left(\frac{u^2}{M^{2d}} \right) \right)}
  = e^{-u+o(1)}.
 \end{equation}
 On the other hand, using $(1-x)^{1-d}-1\geq (d-1)x$ we obtain
 \begin{equation}
  \int_0^{\infty} \one\{u\leq M^d\} e^{-\frac{M^d}{d-1}\left(\left(1-\frac{u}{M^d}\right)^{1-d} -1 \right)} \Di u
  \leq \int_0^\infty e^{-u} \Di u = 1.
 \end{equation}
 Thus \eqref{eq:limnotDn} follows by dominated convergence.
\end{proof}

\subsection{Proof of Theorem~\ref{thm:threshold}, upper bound}\label{ssec:Proofupper}
We are now in position to assemble the proof of the upper bound to the threshold in Theorem~\ref{thm:threshold}.

For this, fix $\gep>0$, and write $i(M_1):=\floor{M_1 n^{(d-1)/d}}$. We choose $M_1>0$ large enough such that
\begin{equation}\label{eq:M_1}
 \lim_{n\to\infty} \IP\big(i(M_1) \notin D_n\big) \geq 1-\gep,
\end{equation}
which is possible due to Lemmas~\ref{lemma:notDn} and \ref{lemma:limnotDn}. Given $M_1$, we choose $M_2>0$ large enough such that
\begin{equation}\label{eq:M_2}
 \lim_{n\to\infty} \IP\Big( \exists i\leq M_2n^{\frac{d-1}{d+1}} : a_i\big(i(M_1) \big) =1 \Big) \geq 1-\gep.
\end{equation}
This is possible due to Theorem~\ref{thm:sources}, applied to the scale $M_1 n^{(d-1)/d}$. To be precise, the number of such vertices converges to a Poisson distribution with parameter $(d+1)^{-1}M_1^{-d}M_2^{d+1}$, obtained from a reparametrization of \eqref{eq:sourcePoi}.

We now employ a coupling, as follows: First reveal $G_{i(M_1)}$. We construct the coupled instance of the DAG $G'_{i(M_1)}$ as an identical copy of $G_{i(M_1)}$. Now suppose that in $G_{i(M_1)}=G'_{i(M_1)}$, there exists a smallest source $i_0\leq M_2 n^{(d-1)/(d+1)}$, otherwise we declare that our coupling fails (and it will be convenient to then set $i_0=n$). Afterwards, we construct $G_n$ from $G_{i(M_1)}$ recursively, and $G'_n$ from $G'_{i(M_1)}$ according to the following rules:
\begin{itemize}
 \item For every edge $k\longrightarrow i_0$ revealed in $G_n$, we draw the arrow $k\longrightarrow i(M_1)$ in $G'_n$.
 \item Conversely, for every edge $k\longrightarrow i(M_1)$ revealed in in $G_n$, we draw the arrow $k\longrightarrow i_0$ in $G'_n$.
 \item All other arrows are identical in $G_n$ and $G'_n$.
\end{itemize}
Conditional on the coupling succeeding, both $G_n$ and $G'_n$ will be $d$-DAGs in the sense that every vertex has outdegree $d$ and all arrows point from a larger to a smaller index. Moreover, $G'_n$ will have the same distribution as $G_n$, since at every step, the targets for newly drawn edges are chosen uniformly at random.

In the case of a successful coupling, we have by construction that $i_0\notin D_n$ if and only if $i(M_1)\notin D'_n$, where $D'_n$ denotes the set of descendants of $n$ in $G'_n$. Therefore
\begin{align}\label{eq:upperbdproof}
 \IP\Big( \exists i\leq M_2n^{\frac{d-1}{d+1}}: i\notin D_n\Big)
 &\geq \IP\big( i_0\notin D_n \cap \text{ Coupling succeeds } \big)\notag\\
 &= \IP\big( i(M_1)\notin D'_n \cap \text{ Coupling succeeds }\big)\notag\\
 &= \IP\big(i(M_1)\notin D'_n\big) \cdot \IP\big( \text{ Coupling succeeds }\big),
\end{align}
where we used that the edge set determining if $i(M_1)\in D'_n$ is disjoint from the edge set that determines if the coupling succeeds (the former uses arrows originating in vertices $>i(M_1)$, the latter arrows originating in $\leq i(M_1)$); hence the two events are independent. Thus
\begin{equation}
 \IP\Big( \exists i\leq M_2n^{\frac{d-1}{d+1}}: i\notin D_n\Big) \geq (1-\gep+o(1))^2
\end{equation}
by applying \eqref{eq:M_1} and \eqref{eq:M_2} to \eqref{eq:upperbdproof}, which provides the desired upper bound for the threshold claim, yielding the second part of \eqref{eq:proofstratthreshold}.



\bibliographystyle{plain}
\bibliography{dDAG}

\end{document}